\documentclass[11pt]{article}

\usepackage{amssymb,amsmath}
\usepackage{enumerate} 
\usepackage{enumitem} 

\newtheorem{theorem}{Theorem}[section]
\newtheorem{lemma}[theorem]{Lemma}

\newenvironment{proof}[1][]%
{\noindent {\setcounter{equation}{0}\it Proof.
}{#1}{}}{\hfill$\Box$\vspace{2ex}}

\def\longbox#1{\parbox{0.85\textwidth}{#1}}

\def\Gbar{\overline{G}}

\begin{document}

\title{Coloring $(P_5,\mbox{bull})$-free graphs}

\author{%
Fr\'ed\'eric Maffray\thanks{CNRS, Laboratoire G-SCOP,
Univ.~Grenoble-Alpes, Grenoble, France. \newline  Partially supported by ANR
project STINT under reference ANR-13-BS02-0007.}}

\date{\today}

\maketitle

\begin{abstract}
We give a polynomial-time algorithm that computes the chromatic number
of any graph that contains no path on five vertices and no bull as an
induced subgraph (where the bull is the graph with five vertices
$a,b,c,d,e$ and edges $ab,bc,cd,be,ce$).

\end{abstract}

\section{Introduction}

For any graph $G$ and integer $k$, a \emph{$k$-coloring} of $G$ is a
mapping $c:V(G)\rightarrow\{1,\ldots,k\}$ such that any two adjacent
vertices $u,v$ in $G$ satisfy $c(u)\neq c(v)$.  A graph is
\emph{$k$-colorable} if it admits a $k$-coloring.  The \emph{chromatic
number} $\chi(G)$ of a graph $G$ is the smallest integer $k$ such that
$G$ is $k$-colorable.  Determining whether a graph is $k$-colorable is
NP-complete for each fixed $k\ge 3$ \cite{GJS,Karp}.

For any integer $\ell$ let $P_\ell$ denote the path on $\ell$ vertices
and $C_\ell$ denote the cycle on $\ell$ vertices.  The \emph{bull} is
the graph with five vertices $a,b,c,d,e$ and edges $ab,bc,cd,be,ce$.
Given a family of graphs ${\cal F}$, a graph $G$ is \emph{${\cal
F}$-free} if no induced subgraph of $G$ is isomorphic to a member of
${\cal F}$; when ${\cal F}$ has only one element $F$ we say that $G$
is $F$-free.  Coloring $P_5$-free graphs is an NP-complete problem, as
proved by Kr\'al' et al.~\cite{KKTW}.  On the other hand, Ho\`ang et
al.~\cite{HKLSS} proved that the problem of $k$-coloring $P_5$-free
graphs is polynomially solvable for every fixed $k$.  The complexity
of coloring ($P_5$,bull)-free graphs is mentioned as an open problem
in \cite{CamHoa} as well as on the Graph Classes website
(http://www.graphclasses.org).  Our main result is the following.

\begin{theorem}\label{thm:main}
There is a polynomial time algorithm that finds the chromatic number
of any $(P_5, \mbox{bull})$-free graph and gives a $\chi(G)$-coloring
of $G$.
\end{theorem}

Let $G$ be a graph.  We say that a vertex $v$ is \emph{complete} to
$S$ if $v$ is adjacent to every vertex in $S$, and that $v$ is
\emph{anticomplete} to $S$ if $v$ has no neighbor in $S$.  For two
sets $S,T\subseteq V(G)$ we say that $S$ is \emph{complete} to $T$ if
every vertex of $S$ is adjacent to every vertex of $T$, and we say
that $S$ is \emph{anticomplete} to $T$ if no vertex of $S$ is adjacent
to any vertex of $T$.  For $S\subseteq V(G)$ we denote by $G[S]$ the
induced subgraph of $G$ with vertex-set $S$.  The complement of $G$ is
denoted by $\overline{G}$.

A coloring of a graph $G$ is a partition of $V(G)$ into stable sets.
A \emph{clique cover} is a partition of $V(G)$ into cliques.  Hence of
a coloring of a graph $G$ is a clique cover of $\Gbar$ and vice-versa.

We let $K_n$ denote the complete graph on $n$ vertices.  The graph
$K_3$ is usually called a \emph{triangle}.  The graph $\overline{P_5}$
is usually called the \emph{house}.  Note that the bull is a
self-complementary graph.  Hence the problem of coloring a
($P_5$,bull)-free graph is equivalent to the problem of finding a
clique cover of a (bull,house)-free graph.  We find it more convenient
to adopt this latter point of view.  So our main result can be
reformulated as follows.

\begin{theorem}\label{thm:main2}
There is a polynomial time algorithm that finds a minimum-size clique
cover in any (bull, house)-free graph.
\end{theorem}

In a graph $G$, two vertices are called \emph{duplicates} if they 
have the same neighbors (in particular, they are not adjacent). 
Given a graph $G$ and a vertex $u$ of $G$, \emph{duplicating} $u$ 
means creating a new vertex $u'$ with the same neighbors as $u$.

A \emph{homogeneous set} in a graph $G$ is a set $S\subseteq V(G)$
such that every vertex in $V(G)\setminus S$ is either complete or
anticomplete to $S$.  A homogeneous set is \emph{proper} if it
contains at least two vertices and is different from $V(G)$.  A
\emph{module} is a homogeneous set $M$ such that every homogeneous set
$S$ satisfies either $S\subseteq M$ or $M\subseteq S$ or $S\cap
M=\emptyset$.  In particular $V(G)$ is a module and every one-vertex
set is a module.  It follows from their definition that the modules
form a ``nested'' family, so their inclusion relation can be
repesented by a tree, and any graph $G$ has at most $2|V(G)|-1$
modules.  The modules of a graph $G$ can be produced by an algorithm
of linear (i.e., $O(|V(G)|+|E(G)|)$) time complexity
\cite{CH,CHPT,MS}.  A module $M$ is \emph{quasi-maximal} if the only
module that contains $M$ strictly is $V(G)$.  So the quasi-maximal
modules form a partition of $V(G)$.  We will use the following results
from \cite{Gallai}.
\begin{lemma}[Gallai \cite{Gallai}]\label{lem:hqm}
Let $G$ be any graph. \\
--- If $M$ is any module of $G$, then the modules of $G[M]$ are the 
modules of $G$ that are contained in $M$. \\
--- If $G$ (resp.~$\overline{G}$) is not connected, then the
quasi-maximal modules of $G$ are the vertex-sets of the components of
$G$ (resp.~of $\overline{G}$). \\
--- If $G$ and $\overline{G}$ are connected, then every proper
homogeneous set of $G$ is included in a quasi-maximal module of $G$.
\end{lemma}

Our algorithm will be based on the following structural result.

\begin{theorem}\label{thm:hbf}
Let $G$ be any (house, bull)-free graph. Then either: \\
--- $G$ has a proper homogeneous set that is not a stable set, or \\
--- $G$ is $C_5$-free and $P_5$-free, or \\
--- $G$ is triangle-free.
\end{theorem}
\begin{proof}
Let $G$ be a (house, bull)-free graph.  Suppose that $G$ does not
satisfy the first or the second item of the theorem, and let us prove
that it satisfies the third.  Hence $G$ contains a $P_5$ or a $C_5$.
So there exist five non-empty and pairwise disjoint subsets $A_1,
\ldots, A_5$ of $V(G)$ such that the following properties hold, with
subscripts modulo~$5$:
\begin{itemize}
\item
For each $i\in\{1,2,3,4\}$, $A_i$ is complete to $A_{i+1}$.
\item
For each $i\in\{1,2,3,4,5\}$, $A_i$ is anticomplete to $A_{i+2}$.
\item
$A_5$ is either complete or anticomplete to $A_1$. 
\end{itemize}
Note that if $A_5$ is complete to $A_1$ the five sets play symmetric
roles.  Let $A=A_1\cup\cdots\cup A_5$.  We choose these sets so that
$A$ is inclusionwise maximal.  Let $B$ be the set of vertices of
$V(G)\setminus A$ that are complete to $A$.  We first claim that:
\begin{equation}\label{no4}
\longbox{For any vertex $v\in V(G)\setminus (A\cup B)$ and any
$i\in\{1,\ldots, 5\}$, $v$ is anticomplete to at least one of $A_i$,
$A_{i+1}$, $A_{i+2}$, $A_{i+3}$.}
\end{equation}
Proof: For each $i\in\{1,...,5\}$, let $a_i$ be a neighbor of $v$ in
$A_i$ (if any) and let $z_i$ be a non-neighbor of $v$ in $A_i$ (if
any).  Suppose that $v$ has neighbors in four sets $A_i$, $A_{i+1}$,
$A_{i+2}$, $A_{i+3}$.  Up to symmetry we may assume that
$i\in\{1,3,4\}$.  If $i=1$, then $v$ is complete to $A_5$, for
otherwise $\{a_1,v,a_3,a_4,z_5\}$ induces either a house or a bull
(depending on the adjacency between $a_1$ and $z_5$).  If $i=3$, then
$v$ is complete to $A_2$, for otherwise $\{a_1,z_2,a_3,a_4,v\}$
induces a house.  If $i=4$, then $v$ is complete to $A_3$, for
otherwise $\{a_1,a_2,z_3,a_4,v\}$ induces a house.  In all cases $v$
is complete to $A_{i-1}$, so $v$ has neighbors in all five sets.
Repeating this argument with each $i$ we obtain that $v\in B$, a
contradiction.  Thus (\ref{no4}) holds.

\medskip

Now we claim that:
\begin{equation}\label{no2}
\longbox{For any vertex $v\in V(G)\setminus (A\cup B)$ and any
$i\in\{1,\ldots, 4\}$, $v$ is anticomplete to at least one of $A_i$
and $A_{i+1}$. Also, if $A_5$ is complete to $A_1$, then $v$ is 
anticomplete to one of $A_1,A_5$.}    
\end{equation}
Proof: For each $i\in\{1,...,5\}$, let $a_i$ be a neighbor of $v$ in
$A_i$ (if any) and let $z_i$ be a non-neighbor of $v$ in $A_i$ (if
any).  Suppose that $v$ has neighbors in two consecutive sets $A_i$
and $A_{i+1}$.  \\
First suppose that $A_5$ is complete to $A_1$.  Up to symmetry, we may
assume that $i=1$.  Then $v$ is complete to $A_5$ or to $A_3$, for
otherwise $\{z_5, a_1,v,a_2, z_3\}$ induces a bull.  By symmetry we
may assume that $v$ is complete to $A_3$; and it follows from
(\ref{no4}) that $v$ has no neighbor in $A_5\cup A_4$.  Moreover $v$
is complete to $A_1$, for otherwise $\{z_1,a_2,v,a_3,z_4\}$ induces a
bull.  But now the sets $A_1, A_2\cup\{v\},A_3,A_4,A_5$ contradict the
maximality of $A$.  \\
Therefore we may assume that $A_5$ is anticomplete to $A_1$.  Up to
symmetry we have $i\in\{1,2\}$.  Suppose that $i=1$.  Suppose that $v$
has a non-neighbor $z_3\in A_3$.  Then $v$ is anticomplete to $A_4$,
for otherwise $\{a_1,a_2,z_3,a_4,v\}$ induces a house; and $v$ is
anticomplete to $A_5$, for otherwise $\{z_3,a_2,a_1,v,a_5\}$ induces a
bull; and $v$ is complete to $A_2$, for otherwise either
$\{v,a_1,z_2,z_3,a_2\}$ induces a house (if $a_2z_2\notin E(G)$) or
$\{v,a_2,z_2,z_3,z_4\}$ induces a bull (if $a_2z_2\in E(G)$).  But now
the sets $A_1\cup\{v\}, A_2,A_3,A_4,A_5$ contradict the maximality of
$A$.  Hence $v$ is complete to $A_3$.  By (\ref{no4}), $v$ has no
neighbor in $A_4\cup A_5$.  Then $v$ is complete to $A_1$, for
otherwise $\{z_1,a_2,v,a_3,z_4\}$ induces a bull.  But now the sets
$A_1,A_2\cup\{v\}, A_3,A_4,A_5$ contradict the maximality of $A$.
Finally suppose that $i=2$.  By the preceding point (the case $i=1$)
we may assume that $v$ is anticomplete to $A_1$.  Then $v$ is complete
to $A_4$, for otherwise $\{z_1,a_2,v,a_3,z_4\}$ induces a bull.  By
(\ref{no4}), $v$ is anticomplete to $A_5$.  By symmetry, $v$ is
complete to $A_2$.  But now the sets $A_1,A_2, A_3\cup\{v\},A_4,A_5$
contradict the maximality of $A$.  Thus (\ref{no2}) holds.

\medskip

Now we claim that:
\begin{equation}\label{bem}
\longbox{$B=\emptyset.$}
\end{equation}
Proof: Suppose that $B\neq\emptyset$.  Let $H$ be the component of
$G\setminus B$ that contains~$A$.  By the hypothesis, $V(H)$ is not a
proper homogeneous set, which implies that there exist non-adjacent
vertices $b\in B$ and $x\in V(H)$.  By the definition of $H$ there is
a shortest path $p_1$-$\cdots$-$p_k$ in $H$ with $p_1\in A$ and
$p_k=x$, and we choose the pair $b,x$ so as to minimize $k$.  We have
$k\ge 2$ since $x\notin A$.  We can pick vertices $a_i\in A_i$ for
each $i\in\{1,\ldots,5\}$ so that $p_2$ has a neighbor in
$\{a_1,...,a_5\}$.  We choose three vertices $u,v,w\in\{a_1,...,a_5\}$
so that: (i) $uv$ is the only edge in $G[u,v,w]$, and (ii) $u$ is the
only neighbor of $p_2$ among them; indeed we can find $u,v,w$ as
follows.  If $A_5$ is complete to $A_1$, then by~(\ref{no2}) and
symmetry we may assume that $p_2$ is adjacent to $a_1$ and has no
neighbor in $\{a_2,a_4,a_5\}$, and we set $u=a_1$, $v=a_2$, $w=a_4$.
Suppose that $A_5$ is anticomplete to $A_1$.  If $p_2$ is adjacent to
$a_1$ or $a_2$, let $\{u,v\}=\{a_1,a_2\}$, and let $w$ be a
non-neighbor of $p_2$ in $\{a_4,a_5\}$ ($w$ exists by (\ref{no2})).
The case when $p_2$ is adjacent to $a_5$ or $a_4$ is symmetric.
Finally if the only neighbor of $p_2$ in $\{a_1,...,a_5\}$ is $a_3$,
then let $u=a_3$, $v=a_2$ and $w=a_5$.  In either case, we see that
$b$ is adjacent to $p_2$, for otherwise $\{p_2,u,v,b,w\}$ induces a
bull.  So $k\ge 3$.  By the minimality of $k$, the vertices
$p_3,...,p_k$ have no neighbor in $A$, and $b$ is adjacent to each of
$p_2,...,p_{k-1}$.  Then $\{p_k,p_{k-1},p_{k-2},b,w\}$ induces a bull,
a contradiction.  Thus (\ref{bem}) holds.

\medskip

Now we claim that:
\begin{equation}\label{ais}
\longbox{For each $i\in\{1,...,5\}$,  $A_i$ is a stable set.}
\end{equation}
Proof: Suppose, up to symmetry, that $A_i$ is not a stable set for
some $i\in\{1,2,3\}$.  So $G[A_i]$ has a component $H$ of size at
least $2$.  By the hypothesis, $V(H)$ is not a homogeneous set, so
there is a vertex $z\in V(G)\setminus V(H)$ and two vertices $x,y\in
V(H)$ such that $z$ is adjacent to $y$ and not to $x$, and since $H$
is connected we may choose $x$ and $y$ adjacent.  By the definition of
$H$ we have $z\notin A_i$.  Since $z$ is adjacent to $y$ and not to
$x$, we have $z\notin A\cup B$.  Pick any $a'\in A_{i+1}$ and $a''\in
A_{i+2}$.  By (\ref{no2}) and since $z$ has a neighbor in $A_i$, $z$
is not adjacent to $a'$.  Then $\{z,y,x,a',a''\}$ induces a bull or a
house (depending on the adjacency between $z$ and $a''$), a
contradiction.  Thus (\ref{ais}) holds.

\medskip

Now we claim that:
\begin{equation}\label{gtf}
\longbox{$G$ is triangle-free.}
\end{equation}
Proof: Suppose that $T=\{u,v,w\}$ is the vertex-set of a triangle in
$G$.  By~(\ref{ais}), the graph $G[A]$ is triangle-free.  Moreover,
by~(\ref{no2}), no triangle of $G$ has two vertices in $A$.  So $T$
contains at most one vertex from $A$.  Note that $G$ is connected, for
otherwise the vertex-set of the component that contains $A$ would be a
proper homogeneous set and not a stable set.  So there is a shortest
path $P$ from $A$ to $T$.  Let $P=p_1$-$\cdots$-$p_k$, with $p_1\in
A$, $p_2, \ldots, p_k\in V(G)\setminus A$, $p_k=u$, $k\ge 1$, and
$v,w\notin A$.  We choose $T$ so as to minimize $k$.  We can pick
vertices $a_i\in A_i$ for each $i\in\{1,...,5\}$ so that, up to
symmetry $p_1=a_i$ for some $i\in\{1,2,3\}$.  Let $p_0=a_{i+1}$.  Let
$U$ be the set of neighbor of $u$, and let $H$ be the component of
$G[U]$ that contains $v$ and $w$.  Since $V(H)$ is not a homogeneous
set, there are vertices $x,y\in V(H)$ and $z\in V(G)\setminus V(H)$
such that $z$ is adjacent to $y$ and not to $x$, and since $H$ is
connected we may choose such $x$ and $y$ adjacent.  By the definition
of $H$, the vertex $z$ is not adjacent to $u$.  If $x$ is adjacent to
$p_{k-1}$, then either $k=1$ and (\ref{no2}) is violated (because $x$
is adjacent to $p_1$ and $p_0$), or $k\ge 2$ and $\{p_{k-1},p_k,x\}$
is a triangle that contradicts the minimality of $k$.  So $x$ is not
adjacent to $p_{k-1}$, and similarly $y$ is not adjacent to $p_{k-1}$.
But then $\{z,y,x,u,p_{k-1}\}$ induces a bull or a house (depending on
the adjacency between $z$ and $p_{k-1}$), a contradiction.  Thus
(\ref{gtf}) holds.  This completes the proof of the theorem.
\end{proof}

In view of the first item of Theorem~\ref{thm:hbf}, we are interested
in the class of ($C_5,P_5,$ house)-free graphs.  The following result
applies to this class.
\begin{theorem}[\cite{CHMW}]\label{thm:4classes}
The chromatic number of any ($C_5,P_5$, house)-free graph can be
determined in linear time.
\end{theorem}
For the sake of completeness we can summarize the proof of this
theorem as follows.  Let $v_1,v_2,...,v_n$ be an ordering of the
vertices of $G$ such that their degrees satisfy $d(v_1)\ge d(v_2)
\ge\dots\ge d(v_n)$.  Apply the greedy coloring algorithm on this
ordering.  Then a minimum coloring is produced.  (Actually the theorem
in \cite{CHMW} is proved for a larger class of graphs.)

Note that the class of ($C_5,P_5,$ house)-free graphs is
self-complementary, so Theorem~\ref{thm:4classes} can also be used to
find a minimum clique cover.

\subsection*{The algorithm}

Now we can describe our algorithm. Let $G$ be any (house, bull)-free 
graph for which we want to find a minimal clique cover. Let $cc(G)$ 
denote the size of a minimal clique cover of $G$. 

\bigskip

\noindent(I) Suppose that $G$ is $P_5$-free and $C_5$-free.  Then we
can use Theorem~\ref{thm:4classes}.

\medskip

\noindent(II) Suppose that $G$ is triangle-free.  Then a clique cover
consists of the edges of a matching $M$ plus a one-vertex clique for
each vertex that is not saturated by $M$, so it has size $|V(G)|-|M|$.
Hence determining a minimum clique cover is equivalent to finding a
maximum matching, which can be done in time $O(|V(G)|^3)$ \cite{MV}.

\medskip

\noindent(III) Suppose that $G$ is not connected.  Recall that in this
case the quasi-maximal modules of $G$ are the vertex-sets of its
components, by Lemma~\ref{lem:hqm}.  It suffices to solve the problem
for each component of $G$ and to take the union of the solutions.
Hence $cc(G)$ is the sum of $cc(H)$ over all components $H$ of $G$.

\medskip
 
\noindent(IV) Suppose that $\overline{G}$ is not connected.  Let
$U_1,...,U_p$ ($p\ge 2$) be the vertex-sets of the components of
$\overline{G}$.  Recall that these sets are the quasi-maximal modules
of $G$, by Lemma~\ref{lem:hqm}.  Hence in $G$ any two such sets are
complete to each other.  We can solve the problem recursively for each
induced subgraph $G[U_i]$.  Let $\{Q_i^1,...,Q_i^{c_i}\}$ be a minimum
clique cover of $G[U_i]$ for each $i$, with $c_i=cc(G[U_i])$, and let
$c=\max\{cc(G[U_i])\mid i=1,...,p\}$.  Let $Q^j= Q_1^j\cup\cdots\cup
Q_p^j$ for all $j=1,...,c$.  Then $\{Q^1,...,Q^c\}$ is a clique cover
of $G$, which shows that $cc(G) = \max\{cc(G[U_i])\mid i=1,...,p\}$.

\medskip

\noindent(V) Finally, suppose that $G$ and $\overline{G}$ are
connected, and (by Theorem~\ref{thm:hbf}) that $G$ has a proper
homogeneous set that is not a stable set.  By Lemma~\ref{lem:hqm},
there is a quasi-maximal module $M$ of $G$ that is not a stable set.
We solve the problem recursively on $G[M]$ and obtain a minimum clique
cover $C_M$ of $G[M]$.  In $G$ we replace $M$ with a stable set $S_M$
of size $|C_M|$.  Let $G_M$ be the resulting graph.  We observe that:
\begin{equation}\label{gm}
\longbox{$G_M$ is (house, bull)-free, and $|V(G_M)| < |V(G)|$.}
\end{equation}
Indeed, $G_M$ is obtained from an induced subgraph of $G$ (the graph
$G\setminus (M\setminus x)$ for any $x\in M$) by duplicating vertices.
Duplication cannot create a house or a bull since these two graphs do
not have duplicate vertices.  Moreover, if $|V(G_M)| = |V(G)|$, then
$|M|=|C_M|$, and we know that $M$ is not a stable set, hence
$cc(G[M])<|C_M|$, a contradiction.  Thus (\ref{gm}) holds.

We do the same for every quasi-maximal module of $G$ that is not a
stable set.  Thus we obtain a graph $G'$ where every proper
homogeneous set is a stable set.  By Theorem~\ref{thm:hbf} we can
obtain a minimum clique cover $C'$ of $G'$ by applying steps (I) or
(II).  For each quasi-maximal module $M$ of $G$ that is not a stable
set, we replace the $j$-th vertex of $S_M$ (in the member of $C'$ that
covers this vertex) with the $j$-th member of $C_M$.  Thus we obtain a
minimum clique cover of $G$.

Note that throughout the execution of the algorithm we only recurse on
modules of $G$.  Hence the number of recursive steps is $O(|V(G)|)$.
So the total complexity is $O(|V(G)|^4)$.

\end{document}